\documentclass[11pt]{amsart}
\usepackage{amsfonts}
\usepackage{amsmath,amsthm,amssymb}
\usepackage{cite}

\newtheorem{theorem}{Theorem}
\theoremstyle{plain}

\newtheorem{corollary}{Corollary}

\newtheorem{definition}{Definition}

\numberwithin{equation}{section}

\allowdisplaybreaks

\begin{document}
\title[polyharmonic mappings ]{On the univalence of polyharmonic mappings}
\author{ Layan El Hajj}
\address{ Mathematics Division , AUD, P.O.Box 28282 , Dubai, UAE}
\email{ lhajj@aud.edu}

\begin{abstract}
A $2p$-times continuously differentiable complex valued function $f=u+iv$ in
a simply connected domain $\Omega $ is \emph{polyharmonic} (or \emph{$p$%
-harmonic}) if it satisfies the \emph{polyharmonic} equation $\bigtriangleup
^{p}F=0.$ Every polyharmonic mapping $f$ can be written as $%
f(z)=\sum_{k=1}^{p}|z|^{2(p-1)}G_{p-k+1}(z),$ where each $G_{p-k+1}$ is
harmonic. In this paper we investigate the univalence of polyharmonic
mappings on linearly connected domains and the relation between univalence
of $f(z)$ and that of $G_{p}(z)$. The notion of stable univalence and logpolyharminc mappings are also considered.
\end{abstract}

\maketitle

\section{Introduction}

A $2p$-times continuously differentiable complex valued function $f=u+iv$ in
a simply connected domain $\Omega $ is \emph{polyharmonic} (or $p$-harmonic
) if it satisfies the \emph{polyharmonic} equation 
\begin{equation*}
\bigtriangleup ^{p}f=0,
\end{equation*}%
where $p\geq 1$ is an integer and $\bigtriangleup $ represents the Laplacian
operator 
\begin{equation*}
\bigtriangleup :=4\dfrac{\partial ^{2}}{\partial z\partial {\overline{z}}}:=%
\dfrac{\partial ^{2}}{\partial x^{2}}+\dfrac{\partial ^{2}}{\partial y^{2}},
\end{equation*}

and 
\begin{equation*}
\bigtriangleup ^{p}f=\underset{p\ }{\underbrace{\bigtriangleup
...\bigtriangleup }}f=\bigtriangleup ^{p-1}\bigtriangleup f.
\end{equation*}

Clearly when $p=1,$ $f$ is harmonic and when $p=2,\,\,f$ is biharmonic. The
properties of univalent harmonic mappings have been studied by many authors
(see \cite{CSS, D,DS,DS2,HS,HS2,J}). One the most fundamental articles on
univalent harmonic mappings is due to Clunie-Shiel and Small (\cite{CSS}).

Biharmonic mappings and their univalence have been investigated recently by
several authors (see \cite{AA, AAK, AAK2,AAK3,CPW1}). Biharmonic functions
arise in many physical situations, in fluid dynamics and elasticity problems
which have many applications in engineering and biology (see \cite%
{HB,L,PP,PGM}). In addition biharmonic mappings are closely related to the
theory of Laguerre minimal surfaces(\cite{AA2}). More recently, properties
of polyharmonic functions are investigated. We refer to (\cite%
{CPW2,CRW,LPW,CRW2,QW}) for many interesting results on polyharmonic
mappings.

If $\Omega \subset \mathbf{\mathbb{C}}$ is a simply connected domain, then
it is easy to see that (see \cite{CPW2}) every polyharmonic mapping $f$ can
be written as 
\begin{equation*}
f(z)=\sum_{k=1}^{p}|z|^{2(k-1)}G_{p-k+1}(z),
\end{equation*}%
where each $G_{p-k+1}$ is harmonic, that is $\bigtriangleup G_{p-k+1}=0$ for 
$k\in \{1,...,p\}$. This is known as the Almansi expansion (see \cite{ACL}).

Throughout we consider polyharmonic functions defined on the unit disk $%
\mathbb{D}=\{z:|z|<1\}$.

\begin{definition}
A domain $\Omega \subset \mathbf{\mathbb{C}}$ \ is linearly connected if
there exists a constant $M<\infty $ \ such that any two points $%
w_{1},w_{2}\in \Omega $ are joined by a path $\gamma ,$ $\gamma \subset
\Omega $, of length $\ell (\gamma )\leq M|w_{1}-w_{2}|$.
\end{definition}

Such a domain is necessarily a Jordan domain, and for piecewise smoothly
bounded domains, linear connectivity is equivalent to the boundary having no
inward-pointing cusps.

In \cite{CH}, Chuaqui and Hermandez, considered the relationship between the
harmonic mapping $f=h+\overline{g}$ and its analytic factor $h$ on linearly
connected domains. They show that if $h\ $is an analytic univalent function,
then every harmonic mapping $f=h+\overline{g}$ with dilatation $|\omega |<c$
is univalent if and only if $h(\mathbb{D})$ is linearly connected.

In \cite{AH}, Abdulhadi and El Hajj showed analogous results for biharmonic
functions. In this paper we generalise these results for polyharmonic
mappings. Moreover some results are obtained for logpolyharmonic mappings.
Recently properties of logbiharmonic and logpolyharmonic mappings have been
investigated (See \cite{CPW2,LPW} ). Many physical problems are modelled by
logbiharmonic mappings particulary those arising from fluid flow theory.

The property of stable univalence is also considered.

\section{Main results}

A complex valued function $f:\Omega \rightarrow \mathbf{\mathbb{C}}$ is said
to belong to the class $C^{1}(\Omega )$ if $\Re f$ and $\Im f$ have
continuous first order partial derivatives in $\Omega $. We denote the
Jacobian of $f$ by 
\begin{equation*}
J_{f}=|f_{z}|^{2}-|f_{\overline{z}}|^{2}.
\end{equation*}

\bigskip We also denote 
\begin{equation*}
\lambda _{f}=|f_{z}|-|f_{\overline{z}}|,
\end{equation*}

and 
\begin{equation*}
\Lambda _{f}=|f_{z}|+|f_{\overline{z}}|,
\end{equation*}

We then have 
\begin{equation*}
J_{f}=\lambda _{f}.\Lambda _{f}
\end{equation*}

We first start by noting that the results in \cite{AH}, hold for functions
of the form $f(z) = |z|^{2(p-1)} G(z) + K(z)$, where $G$ is not necessarily
harmonic. (If G is harmonic these classify as a special family of
polyharmonic mappings).

\begin{theorem}
Let $f(z)=|z|^{2(p-1)}G(z)+K(z),$ where $\ G \in C^1(\mathbb{D})$ (not
necessarily harmonic) and $K$ harmonic. If $K$ is univalent and $K(\mathbb{D}%
)\ $is a linearly connected domain with constant $M$, and if%
\begin{equation*}
\frac{2(p-1)|G|+\Lambda _{G}}{\left\vert \lambda _{K}\right\vert }<\frac{1}{M%
},
\end{equation*}%
then $\ f(z)$ is univalent. Moreover, if 
\begin{equation*}
\frac{2(p-1)|G|+\Lambda _{G}}{\left\vert \lambda _{K}\right\vert }\leq C <%
\frac{1}{M},
\end{equation*}%
then $f(\mathbb{D})$ is a linearly connected domain
\end{theorem}

\begin{proof}
Let $H(z)=|z|^{2(p-1)}G(z).$ We define%
\begin{equation*}
\varphi =H\circ K^{-1}.
\end{equation*}%
Given $w\ \epsilon \ K(\mathbb{D}),$ we claim $w+\varphi (w)$ is univalent.

Assume $w+\varphi (w)$ is not univalent, then there exists $w_{1}\neq w_{2}$
such that 
\begin{equation*}
\varphi (w_{2})-\varphi(w_{1})=w_{1}-w_{2}.
\end{equation*}
Let $\gamma $ be \ a path in $K(\mathbb{D})$ joining $w_{1},w_{2}\ $ such
that $l(\gamma )\leq M|w_{2}-w_{1}|.$ Then 
\begin{equation*}
|\varphi (w_{2})-\varphi (w_{1})|\leq \left\vert \int\nolimits_{\gamma
}\varphi _{w}dw+\varphi _{\overline{w}}d\overline{w}\right\vert \leq
\int\nolimits_{\gamma }\left( |\varphi _{w}|+|\varphi _{\overline{w}%
}|\right) |dw|.
\end{equation*}

But 
\begin{equation*}
\varphi _{w}=H_{z}(K^{-1})_{w}+H_{\overline{z}}(\overline{K^{-1}})_{w}
\end{equation*}

\begin{equation*}
\varphi _{\overline{w}}=H_{z}(K^{-1})_{\overline{w}}+H_{\overline{z}}(%
\overline{K^{-1}})_{\overline{w}}.
\end{equation*}

Differentiating $K^{-1}(K(z))=z,$ we show that 
\begin{eqnarray*}
(K^{-1})_{w} K_z + (K^{-1})_{\overline{w}}(\overline{K}_z)&=&1 \\
(K^{-1})_{w} K_{\overline{z}} + (K^{-1})_{\overline{w}}(\overline{K}_{%
\overline{z}})&=&0 \\
\end{eqnarray*}

Solving above system we get, 
\begin{equation*}
(K^{-1})_{w}=\frac{\overline{K}_{\overline{z}}}{|K_z|^2 - |K_{\overline{z}%
}|^2},
\end{equation*}
and, 
\begin{equation*}
(K^{-1})_{\overline{w}}=\frac{-K_{\overline{z}}}{|K_z|^2 - |K_{\overline{z}%
}|^2}.
\end{equation*}

It follows 
\begin{eqnarray*}
|\varphi _{w}|+|\varphi _{\overline{w}}| &\leq & \frac{|H_{z}|}{|K_z|^2 -
|K_{\overline{z}}|^2} ( |\overline{K}_{\overline{z}}|+|K_{\overline{z}}|)+%
\frac{|H_{\overline{z}}|}{J_K} ( |\overline{K}_{\overline{z}}|+|K_{\overline{%
z}}|) \\
&=& \frac{|H_{z}|+|H_{\overline{z}}|}{\left||K_{z}|-|K_{\overline{z}}|\right|%
}
\end{eqnarray*}

where $z=K^{-1}(w)\in D.$ But 
\begin{equation*}
H_{z} = (p-1)z^{p-2}\overline{z}^{p-1}G + |z|^{2(p-1)} G_{z},
\end{equation*}
and 
\begin{equation*}
H_{\overline{z}} = (p-1){\overline{z}}^{p-2}z^{p-1}G + |z|^{2(p-1)} G_{%
\overline{z}},
\end{equation*}

hence, 
\begin{equation*}
|\varphi (w_{2})-\varphi (w_{1})|\leq \int\nolimits_{\gamma }\underset{%
\mathbb{D}}{\sup }\frac{2(p-1)|G|+|G_{z}|+|G_{\overline{z}}|}{\left\vert
|K_{z}|-|K_{\overline{z}}|\right\vert }|dw|<\frac{1}{M}l(\gamma
)<|w_{2}-w_{1}|
\end{equation*}%
\ which is a contradiction. Therefore $f(z)$ \ is univalent.

For the second part of theorem, given $w$ $\in \Omega =K(\mathbb{D}),$ we
let $\Psi (w)=w+\varphi (w),$where $\varphi =H\circ K^{-1},$ and $%
H=|z|^{2(p-1)}G.$ Since $K$ is univalent, we may look at $R=f(\mathbb{D})$,
as the image of $\Omega =K(\mathbb{D})\ $under the mapping $\Psi ,$ and we
show $\Psi (\Omega )$ is linearly connected. Let $\varsigma _{1}=\Psi
(w_{1}),\varsigma _{2}=\Psi (w_{2})$, $w_{1},w_{2}\in \Omega .$ Since $K(%
\mathbb{D})\ $is a linearly connected domain, then there exists a curve $%
\gamma \subset \Omega $ satisfying $l(\gamma )\leq M|w_{2}-w_{1}|.$

Let $\Gamma =\Psi (\gamma ).$

We have showed that 
\begin{equation*}
|\varphi _{w}|+|\varphi _{\overline{w}}|\leq \frac{2(p-1)|G|+\Lambda _{G}}{%
\left\vert \lambda _{K}\right\vert }<C,
\end{equation*}%
It follows 
\begin{equation*}
|\psi _{w}|+|\psi _{\overline{w}}|\leq 1+|\varphi _{w}|+|\varphi _{\overline{%
w}}|<1+C.
\end{equation*}

Hence we have, 
\begin{equation*}
l(\Gamma )=\int\nolimits_{\Gamma }d\varsigma \leq \int\nolimits_{\gamma
}\left( |\psi _{w}|+|\psi _{\overline{w}}|\right) dw<(1+C)l(\gamma )\leq
(1+C)M|w_{2}-w_{1}|.
\end{equation*}

But, 
\begin{eqnarray*}
|\varsigma _{1}-\varsigma _{2}| &=&|w_{1}-w_{2}+\varphi (w_{1})-\varphi
(w_{2})|\geq |w_{1}-w_{2}|-|\varphi (w_{1})-\varphi (w_{2})| \\
&\geq &|w_{1}-w_{2}|-\int\nolimits_{\gamma }\left( |\varphi _{w}|+|\varphi _{%
\overline{w}}|\right) dw \\
&>&|w_{1}-w_{2}|-Cl(\gamma )\geq (1-CM)|w_{1}-w_{2}|
\end{eqnarray*}

It follows, 
\begin{equation*}
l(\Gamma )\leq \frac{(1+C)M}{1-CM}|\varsigma _{1}-\varsigma _{2}|
\end{equation*}

and so $f(\mathbb{D})$ is linearly connected with constant $\frac{(1+C)M}{%
1-CM}.$
\end{proof}

In particular, for the special case where $p=2$ we have the following
corollary :

\begin{corollary}
Let $f(z)=|z|^{2}G(z)+K(z),$ where $\ G \in C^1(D)$ ( not necessarily
harmonic). If $K$ is univalent harmonic and $K(\mathbb{D})\ $is a linearly
connected domain with constant $M$, and if%
\begin{equation*}
\frac{2|G|+\Lambda _{G}}{\left\vert \lambda _{K}\right\vert }<\frac{1}{M},
\end{equation*}%
then $f(z)$ is univalent. Moreover, if 
\begin{equation*}
\frac{2|G|+\Lambda _{G}}{\left\vert \lambda _{K}\right\vert } \leq C <\frac{1%
}{M},
\end{equation*}%
then $f(\mathbb{D})$ is a linearly connected domain.
\end{corollary}

We are now ready to prove the theorem for polyharmonic functions. In fact,
we will state the theorem in its most general form and the polyharmonic
functions will be a special case.

\begin{theorem}
Let 
\begin{equation*}
f(z)=\sum_{k=1}^{p} |z|^{2(k-1)}G_{p-k+1}(z)
\end{equation*}
where $G_{p-k+1} \in C^1(D)$ (not necessarily harmonic) $k=\{2,...,p\} $ (In
particular $f$ is polyharmonic). If $G_p$ is univalent harmonic and $G_p(%
\mathbb{D})$ is a linearly connected domain with constant $M$, and if%
\begin{equation*}
\frac{\sum_{k=1}^{p-1}2k|G_{p-k}|+(k-1)\Lambda _{G_{p-k}}}{\left\vert
\lambda _{G_{p}}\right\vert }<\frac{1}{M},
\end{equation*}%
then $\ f(z)$ is univalent. Moreover, if 
\begin{equation*}
\frac{\sum_{k=1}^{p-1}2k|G_{p-k}|+(k-1)\Lambda _{G_{p-k}}}{\left\vert
\lambda _{G_{p}}\right\vert }\leq C <\frac{1}{M},
\end{equation*}%
then $f(\mathbb{D})$ is a linearly connected domain.
\end{theorem}

\begin{proof}
We rewrite $f(z)$ as $f(z) = |z|^{2} G(z) + G_p (z),$ where 
\begin{equation*}
G(z) = \sum_{k=1}^{p-1}|z|^{2(k-1)} G_{p-k}(z).
\end{equation*}

Since 
\begin{equation*}
G_{z} = \sum_{k=1}^{p-1}(k-1)z^{k-2}\overline{z}^{k-1} G_{p-k}
+|z|^{2(k-1)}(G_{p-k})_{z}
\end{equation*}
and 
\begin{equation*}
G_{\overline{z}} = \sum_{k=1}^{p-1}(k-1)z^{k-1}\overline{z}^{k-2} G_{p-k}
+|z|^{2(k-1)}(G_{p-k})_{\overline{z}}
\end{equation*}
We get that 
\begin{equation*}
\frac{2|G|+|G_{z}|+|G_{\overline{z}} |}{\left\vert \lambda
_{G_{p}}\right\vert } \leq\frac{%
\sum_{k=1}^{p-1}2k|G_{p-k}|+(k-1)(|(G_{p-k})_{z}|+|(G_{p-k})_{\overline{z}}
|)}{\left\vert \lambda _{G_{p}}\right\vert } <\frac{1}{M}.
\end{equation*}
(or for the proof of the second part of the theorem $\leq C <\frac{1}{M}.$)

Therefore by Corollary 1, we get that $f$ is univalent and $f(\mathbb{D})$
is a linearly connected domain.
\end{proof}

The following corollary is deduced as a special case of theorem 2.

\begin{corollary}
Let 
\begin{equation*}
f(z)=\sum_{k=1}^{p} |z|^{2(k-1)}G_{p-k+1}(z)
\end{equation*}
where $G_{p-k+1} \in C^1(D)$ (not necessarily harmonic) $k=\{2,...,p\} $ (In
particular $f$ is polyharmonic). If $G_p$ is univalent harmonic and $G_p(%
\mathbb{D})$ is a a convex domain, and if%
\begin{equation*}
\frac{\sum_{k=1}^{p-1}2k|G_{p-k}|+(k-1)\Lambda _{G_{p-k}}}{\left\vert\lambda
_{G_{p}}\right\vert }<1,
\end{equation*}%
then $\ f(z)$ is univalent. Moreover, if 
\begin{equation*}
\frac{\sum_{k=1}^{p-1}2k|G_{p-k}|+(k-1)\Lambda _{G_{p-k}}}{\left\vert
\lambda _{G_{p}}\right\vert }\leq C <1,
\end{equation*}%
then $f(\mathbb{D})$ is a convex domain.
\end{corollary}

In \cite{AH}, the following theorem that allows to conclude the univalence
of $K$ from the univalence of $f$ was proved:

\begin{theorem}[\protect\cite{AH},theorem 3]
Let $f(z)=|z|^{2}G(z)+K(z)$ be a biharmonic function in the unit disk $%
\mathbb{D}$. Suppose $f$ is univalent and $f(\mathbb{D})\ $is a linearly
connected domain \ with constant $M$ and satisfies 
\begin{equation*}
\frac{2|G|+\Lambda _{G}}{\left\vert \lambda _{f}\right\vert }<\frac{1}{M},
\end{equation*}%
then $K(z)$ is univalent. Moreover, if 
\begin{equation*}
\frac{2|G|+\Lambda _{G}}{\left\vert \lambda _{f}\right\vert }\leq C<\frac{1}{%
M},
\end{equation*}%
then $K(\mathbb{D})$ is a linearly connected domain.
\end{theorem}

More generally, we have

\begin{theorem}
Let $f(z)=|z|^{2(p-1)}G(z)+K(z),$ where $\ G\in C^{1}(\mathbb{D})$ (not
necessarily harmonic) and $K$ harmonic. If $K$ is univalent and $K(\mathbb{D}%
)\ $is a linearly connected domain with constant $M$, and if%
\begin{equation*}
\frac{2(p-1)|G|+\Lambda _{G}}{\left\vert \lambda _{f}\right\vert }<\frac{1}{M%
},
\end{equation*}%
then $\ k(z)$ is univalent. Moreover, if 
\begin{equation*}
\frac{2(p-1)|G|+\Lambda _{G}}{\left\vert \lambda _{f}\right\vert }\leq C<%
\frac{1}{M},
\end{equation*}%
then $k(\mathbb{D})$ is a linearly connected domain
\end{theorem}

\begin{proof}
Let $H(z)=|z|^{2(p-1)}G(z).$ We define%
\begin{equation*}
\varphi =H\circ f^{-1}.
\end{equation*}%
Given $w\ \epsilon \ f(\mathbb{D}),$ we claim $\psi (w)=w-\varphi (w)$ is
univalent.

Assume $w-\varphi (w)$ is not univalent, then there exists $w_{1}\neq w_{2}$
such that 
\begin{equation*}
\varphi (w_{1})-\varphi (w_{2})=w_{1}-w_{2}.
\end{equation*}%
Let $\gamma $ be \ a path in $K(\mathbb{D})$ joining $w_{1},w_{2}\ $ such
that $l(\gamma )\leq M|w_{2}-w_{1}|.$ We proceed as in the prrof of Theorem
1 to show that 
\begin{eqnarray*}
|\varphi (w_{1})-\varphi (w_{2})| &\leq &\left\vert \int\nolimits_{\gamma
}\varphi _{w}dw+\varphi _{\overline{w}}d\overline{w}\right\vert  \\
&\leq &\int\nolimits_{\gamma }\left( |\varphi _{w}|+|\varphi _{\overline{w}%
}|\right) |dw| \\
&\leq &\int\nolimits_{\gamma }\underset{\mathbb{D}}{\sup }\frac{%
2(p-1)|G|+|G_{z}|+|G_{\overline{z}}|}{\left\vert |f_{z}|-|f_{\overline{z}%
}|\right\vert }|dw| \\
&<&\frac{1}{M}l(\gamma ) \\
&<&|w_{2}-w_{1}|
\end{eqnarray*}

\ which is a contradiction. Therefore $K(z)$ \ is univalent.

For the second part of theorem, given $w$ $\in \Omega =f(\mathbb{D}),$ we
let $\psi (w)=w-\varphi (w)\ $and we show $\psi (\Omega )$ is linearly
connected., where $\Omega =f(\mathbb{D}).$ Let $\varsigma _{1}=\psi
(w_{1}),\varsigma _{2}=\psi (w_{2})$, $w_{1},w_{2}\in \Omega .$ Since $f(%
\mathbb{D})\ $is a linearly connected domain, then there exists a curve $%
\gamma \subset \Omega $ satisfying $l(\gamma )\leq M|w_{2}-w_{1}|.$

Let $\Gamma =\psi (\gamma ).$We proceed as in the proof of Theorem 1 and we
show that

\begin{equation*}
l(\Gamma )\leq \frac{(1+C)M}{1-CM}|\varsigma _{1}-\varsigma _{2}|
\end{equation*}

and so $k(\mathbb{D})$ is linearly connected with constant $\dfrac{(1+C)M}{%
1-CM}.$
\end{proof}

In a similar fashion as in the proof of Theorem 2, we generalize the theorem
3 in \cite{AH} to the following :

\begin{theorem}
Let 
\begin{equation*}
f(z)=\sum_{k=1}^{p}|z|^{2(k-1)}G_{p-k+1}(z)
\end{equation*}%
where $\ G_{p-k+1}\in C^{1}(\mathbb{D})$ (not necessarily harmonic), $%
k=\{2,...,p\}.$ Suppose $f$ is univalent and $f(\mathbb{D})\ $is a linearly
connected domain \ with constant $M$ and satisfies%
\begin{equation*}
\frac{\sum_{k=1}^{p-1}2k|G_{p-k}|+(2(k-1))(|\Lambda _{G_{p-k}}|)}{\left\vert
\lambda _{f}\right\vert }<\frac{1}{M},
\end{equation*}%
then $G_{p}(z)$ is univalent.Moreover, if 
\begin{equation*}
\frac{\sum_{k=1}^{p-1}2k|G_{p-k}|+(2(k-1))(|\Lambda _{G_{p-k}}|)}{\left\vert
\lambda _{f}\right\vert }\leq C<\frac{1}{M},
\end{equation*}%
then $G_{p}(\mathbb{D})$ is a linearly connected domain.
\end{theorem}

\begin{proof}
We rewrite $f(z)$ as $f(z) = |z|^{2} G(z) + G_p (z),$ where 
\begin{equation*}
G(z) = \sum_{k=1}^{p-1}|z|^{2(k-1)} G_{p-k}(z).
\end{equation*}

Since 
\begin{equation*}
G_{z}=\sum_{k=1}^{p-1}(k-1)z^{k-2}\overline{z}%
^{k-1}G_{p-k}+|z|^{2(k-1)}(G_{p-k})_{z}
\end{equation*}%
and 
\begin{equation*}
G_{\overline{z}}=\sum_{k=1}^{p-1}(k-1)z^{k-1}\overline{z}%
^{k-2}G_{p-k}+|z|^{2(k-1)}(G_{p-k})_{\overline{z}}
\end{equation*}%
We get that 
\begin{equation*}
\frac{2|G|+|G_{z}|+|G_{\overline{z}}|}{\left\vert \lambda _{f}\right\vert }%
\leq \frac{\sum_{k=1}^{p-1}2k|G_{p-k}|+(k-1)(|(G_{p-k})_{z}|+|(G_{p-k})_{%
\overline{z}}|)}{\left\vert \lambda _{f}\right\vert }<\frac{1}{2M}.
\end{equation*}

Therefore by theorem 3 in \cite{AH} , $f$ is univalent.
\end{proof}

The case where $f$ is convex is a corollary :

\begin{corollary}
Let 
\begin{equation*}
f(z)=\sum_{k=1}^{p}|z|^{2(k-1)}G_{p-k+1}(z)
\end{equation*}%
where $\ G_{p-k+1}\in C^{1}(\mathbb{D})$ (not necessarily harmonic), $%
k=\{2,...,p\}.$ Suppose $f$ is univalent ,$f((\mathbb{D})$ is a convex
domain. and satisfies%
\begin{equation*}
\frac{\sum_{k=1}^{p-1}2k|G_{p-k}|+(2(k-1))(|\lambda _{G_{p-k}}|)}{\left\vert
\lambda _{f}\right\vert }<\frac{1}{2},
\end{equation*}%
then $G_{p}(z)$ is univalent.Moreover, if 
\begin{equation*}
\frac{\sum_{k=1}^{p-1}2k|G_{p-k}|+(2(k-1))(|\lambda _{G_{p-k}}|)}{\left\vert
\lambda _{f}\right\vert }\leq C<\frac{1}{2},
\end{equation*}%
then $G_{p}(\mathbb{D})$ is a convex domain.
\end{corollary}

We next prove a stable univalence property for polyharmonic mappings onto
linearly connected domains.

\begin{theorem}
Let $f(z)=|z|^{2(p-1)}G(z)+K(z),$ where $\ G\in C^{1}(\mathbb{D})$ (not
necessarily harmonic) and $K$ harmonic. If $f$ is univalent and $f(\mathbb{D}%
)\ $is a linearly connected domain with constant $M$, and if%
\begin{equation*}
\frac{2(p-1)|G|+\Lambda _{G}}{\left\vert \lambda f\right\vert }<\frac{1}{2M},
\end{equation*}%
then $f_{a}(z)=a|z|^{2(p-1)}G(z)+K(z)$ is univalent for any $a$, such that $%
|a|<1$.Moreover, if 
\begin{equation*}
\frac{2(p-1)|G|+\Lambda _{G}}{\left\vert \lambda _{f}\right\vert }\leq C<%
\frac{1}{2M},
\end{equation*}%
then $f_{a}(\mathbb{D})$ is a a linearly connected domain.
\end{theorem}

\begin{proof}
We note that

\begin{eqnarray*}
f_{a}(z) &=&a|z|^{2(p-1)}G(z)+K(z) \\
&=&a|z|^{2(p-1)}G(z)+f(z)-|z|^{2(p-1)}G(z) \\
&=&f(z)+(a-1)|z|^{2(p-1)}G(z)
\end{eqnarray*}

Suppose $f_{a}(z)$ is not univalent, then there exists $z_{1}\neq z_{2}$
such that 
\begin{equation*}
f(z_{2})-f(z_{1})=(1-a)(H(z_{2})-H(z_{1})),
\end{equation*}%
where $H(z)=|z|^{2(p-1)}G(z).$

Let $w=f(z)$ and $\varphi =H\circ f^{-1},$ we get $w_{2}-w_{1}=(1-a)(\varphi
(w_{2})-\varphi (w_{1}))$.So $\left\vert w_{2}-w_{1}\right\vert \leq
2\left\vert \varphi (w_{2})-\varphi (w_{1})\right\vert .$

As in the proof of theorem 1, we have 
\begin{equation*}
|\varphi (w_{2})-\varphi (w_{1})|\leq \int\nolimits_{\gamma }\underset{D}{%
\sup }\frac{2(p-1)|G|+|G_{z}|+|G_{\overline{z}}|}{\left\vert |f_{z}|-|f_{%
\overline{z}}|\right\vert }|dw|<\frac{1}{2M}l(\gamma )<\frac{|w_{2}-w_{1}|}{2%
}
\end{equation*}%
\ which is a contradiction.

Therefore $f_{a}(z)$ \ is univalent.

Next we assume$\dfrac{2(p-1)|G|+\Lambda _{G}}{\left\vert \lambda
_{f}\right\vert }\leq C<\frac{1}{2M}.$We let $\psi (w)=w+(a-1)\varphi (w)\ $%
and we show $\psi (\Omega )$ is linearly connected.where $\Omega =f(\mathbb{D%
}).$ Let $\varsigma _{1}=\psi (w_{1}),\varsigma _{2}=\psi (w_{2})$, $%
w_{1},w_{2}\in \Omega .$ Since $f(\mathbb{D})\ $is a linearly connected
domain, then there exists a curve $\gamma \subset \Omega $ satisfying $%
l(\gamma )\leq M|w_{2}-w_{1}|.$

Let $\Gamma =\psi (\gamma ).$We proceed as in the proof of Theorem 1 and we
show that 

Hence we have, 
\begin{equation*}
l(\Gamma )\leq \int\nolimits_{\gamma }\left( |\psi _{w}|+|\psi _{\overline{w}%
}|\right) dw<(1+|a-1|C)l(\gamma )\leq (1+2C)M|w_{2}-w_{1}|.
\end{equation*}

But, 
\begin{eqnarray*}
|\varsigma _{1}-\varsigma _{2}| &\geq &|w_{1}-w_{2}|-|1-a||\varphi
(w_{1})-\varphi (w_{2})| \\
&>&|w_{1}-w_{2}|-2Cl(\gamma ) \\
&\geq &(1-2CM)|w_{1}-w_{2}|
\end{eqnarray*}

It follows, 
\begin{equation*}
l(\Gamma )\leq \frac{(1+2C)M}{1-2CM}|\varsigma _{1}-\varsigma _{2}|
\end{equation*}

and so $k(\mathbb{D})$ is linearly connected with constant $\dfrac{(1+2C)M}{%
1-2CM}.$

since $f$ is univalent and $f(\mathbb{D})\ $is a linearly connected domain ,
it follows by theorem 3 $K(\mathbb{D})$ is a linearly connected domain.
Hence by theorem 1, $f_{a}(\mathbb{D})$ is a linearly connected domain.
\end{proof}

\begin{corollary}
Let $f(z)=|z|^{2(p-1)}G(z)+K(z),$ where $\ G\in C^{1}(\mathbb{D})$ (not
necessarily harmonic) and $K$ harmonic. If $f$ is univalent and $f(\mathbb{D}%
)\ $is a convex domain, and if%
\begin{equation*}
\frac{2(p-1)|G|+\Lambda _{G}}{\left\vert \lambda _{f}\right\vert }<\frac{1}{2%
},
\end{equation*}%
then $f_{a}(z)=a|z|^{2(p-1)}G(z)+K(z)$ is univalent for any $a$, such that$%
|a|<1$.Moreover, if 
\begin{equation*}
\frac{2(p-1)|G|+\Lambda _{G}}{\left\vert \lambda _{f}\right\vert }\leq C<%
\frac{1}{2},
\end{equation*}%
then $f_{a}(\mathbb{D})$ is a a convex domain.
\end{corollary}

\section{$\log p $-harmonic mappings}

A log harmonic mapping defined on $\mathbb{D}$ is a solution of the non
linear elliptic partial differential equation 
\begin{equation*}
\overline{f_{\overline{z}}}=\left( \frac{\mu \overline{f}}{f}\right)
f_{z},\,\,\,\,\,\,\,\,\,\,\,f(0)=0
\end{equation*}%
where the second dilation $\mu $ is analytic in $\mathbb{D}$ such that $|\mu
(z)|<1.$ In general the solution of this equation is not necessarily
univalent. For example,$f(z)=|z|^{4}z^{4}\ $is logharmonic but not univalent
in \newline
D.We say that $f$ $\log p$-harmonic if $\log f$ is p-harmonic. Throughout
\textquotedblleft\ $\log $ \textquotedblright\ denotes the principal branch
of the logarithm. It can be easily shown that every $\log p$-harmonic
function in a simply connected domain $\Omega $ has the form 
\begin{equation*}
f(z)=\Pi _{k=1}^{p}(g_{p-k+1}(z))^{|z|^{2(k-1)}},
\end{equation*}%
where all $g_{p-k+1}(z)$ are non vanishing log harmonic mappings in $\Omega $
for $k=\{1,...,p\}.$

\begin{theorem}
Let 
\begin{equation*}
f(z) = \Pi_{k=1}^{p} (g_{p-k+1}(z))^{|z|^{2(k-1)}},
\end{equation*}
where all $g_{p-k+1}(z)$ are non vanishing log harmonic mappings in $\mathbb{%
D}$, $k=\{1,...,p\}.$ If $\log g_p$ is univalent and $\log g_p (\mathbb{D})$
is a linearly connected domain with constant $M$, and if%
\begin{equation*}
\frac{|g_{p}|\sum_{k=1}^{p-1}2k|g_{p-k}||\log g_{p-k}|+(k-1)\Lambda
_{g_{p-k}}}{\left\vert \lambda _{g_{p}}\right\vert }<\frac{1}{M},
\end{equation*}

then $\ f(z)$ is univalent.Moreover, if 
\begin{equation*}
\frac{|g_{p}|\sum_{k=1}^{p-1}2k|g_{p-k}||\log g_{p-k}|+(k-1)\Lambda
_{g_{p-k}}}{\left\vert \lambda _{g_{p}}\right\vert }\leq C<\frac{1}{M},
\end{equation*}%
then $G_{p}(\mathbb{D})$ is a linearly connected domain.
\end{theorem}

\begin{proof}
Define $F = \log f$ and $G_{p-k} = \log g_{p-k}$ for all $g_{p-k}$, $%
k=\{1,...,p\}.$ Then $F$ is $p$-harmonic and all $G_{p-k}$ are harmonic.
Moreover, $G_p$ is univalent, $G_p(D)$ is a linearly connected domain with
constant $M$, and 
\begin{eqnarray*}
&&\frac{\sum_{k=1}^{p-1}2k|G_{p-k}|+(k-1)(|(G_{p-k})_{z}|+|(G_{p-k})_{%
\overline{z}} |)} {\left||(G_{p})_{z}|-|(G_{p})_{\overline{z}}|\right|} \\
&= & \frac{|g_p|\sum_{k=1}^{p-1}2k |g_{p-k}||\log
g_{p-k}|+(k-1)(|(g_{p-k})_{z}|+|(g_{p-k})_{\overline{z}} |)} {%
\left||(g_{p})_{z}|-|(g_{p})_{\overline{z}}|\right|} \\
&<& \frac{1}{M},
\end{eqnarray*}

and so $\ F(z)$ is univalent by Theorem 2. Therefore $f$ is univalent.
\end{proof}

\begin{theorem}
Let 
\begin{equation*}
f(z)=\Pi _{k=1}^{p}(g_{p-k+1}(z))^{|z|^{2(k-1)}},
\end{equation*}%
where all $g_{p-k+1}(z)$ are non vanishing log harmonic mappings in $\mathbb{%
D}$, $k=\{1,...,p\}.$ If $\log g_{p}$ is univalent and $\log g_{p}(\mathbb{D}%
)$ is convex, and if%
\begin{equation*}
\frac{|g_{p}|\sum_{k=1}^{p-1}2k|g_{p-k}||\log g_{p-k}|+(k-1)\Lambda
_{g_{p-k}}}{\left\vert \lambda _{g_{p}}\right\vert }<1,
\end{equation*}

then $\ f(z)$ is univalent.Moreover, if 
\begin{equation*}
\frac{|g_{p}|\sum_{k=1}^{p-1}2k|g_{p-k}||\log g_{p-k}|+(k-1)\Lambda
_{g_{p-k}}}{\left\vert \lambda _{g_{p}}\right\vert }\leq C<1,
\end{equation*}%
then $G_{p}(\mathbb{D})$ is a convex domain.
\end{theorem}

\end{document}